\documentclass[runningheads]{llncs}

\usepackage[T1]{fontenc}

\usepackage{newtxtext}
\usepackage[varvw]{newtxmath}

\usepackage{mathtools}
\usepackage{graphicx}
\usepackage{tikz}
\usepackage{chngcntr}

\usepackage{url}
\usepackage{color}

\usepackage{hyperref}

\hypersetup{
    colorlinks=true,
    linkcolor=blue,
    citecolor=blue,
    urlcolor=blue
}

\tikzset{
  inS/.style={
    circle,
    fill=black,
    inner sep=2.2pt
  },
  notS/.style={
    circle,
    draw=black,
    fill=white,
    inner sep=2.2pt,
    line width=0.8pt
  },
  Hedge/.style={
    line width=1.5pt
  },
  vtxlbl/.style={
    font=\small
  }
}

\newcommand{\Set}[1]{\left\{ #1 \right\}}
\newcommand{\card}[1]{\left| #1 \right|}
\newcommand{\Union}{\bigcup}

\newcommand{\union}{\cup}
\newcommand{\intersect}{\cap}
\newcommand{\Hull}[2][]{\operatorname{Hull}_{#1}(#2)}
\newcommand{\Itr}[1]{I_{\triangle}(#1)}
\newcommand{\defn}{\coloneq}
\newcommand{\Tri}{%
  \mathord{\mathchoice
    {\raisebox{-0.05ex}{\scalebox{1.20}{$\displaystyle\triangle$}}}
    {\raisebox{-0.05ex}{\scalebox{1.20}{$\textstyle\triangle$}}}
    {\mkern-2mu
     \raisebox{-0.03ex}{\scalebox{1.10}{$\scriptstyle\triangle$}}
     \mkern-2mu}
    {\mkern-1mu
     \raisebox{-0.02ex}{\scalebox{1.05}{$\scriptscriptstyle\triangle$}}
     \mkern-1mu}%
  }%
}

\title{Tight Bounds on the Carath\'eodory and Exchange Numbers in \(\Tri\)-Convexity}

\author{
Vishnu Kumar\orcidID{0009-0005-8914-2340}
\and
Brahadeesh Sankarnarayanan\orcidID{0000-0001-9191-1253}
}

\authorrunning{V. Kumar and B. Sankarnarayanan}

\institute{
Department of Mathematics,
Indian Institute of Technology Jodhpur,
Jodhpur 342030, Rajasthan, India\\
\email{p25ma0014@iitj.ac.in}, \email{brahadeesh@iitj.ac.in}\\
\url{https://brahadeesh1994@github.io}
}

\begin{document}

\maketitle

\begin{abstract}
The \emph{\(\Tri\)-convexity space} on a finite, simple graph \(G = (V,E)\) is the collection \(\mathcal{C}\) of subsets \(S \subseteq V(G)\) such that whenever \(x \in V(G)\) forms a triangle with two vertices in \(S\), we have \(x \in S\).
The members of \(\mathcal{C}\) are called \emph{convex} sets, and the \emph{convex hull} of a set \(S \subseteq V(G)\), denoted \(\Hull{S}\), is the smallest member of \(\mathcal{C}\) that contains \(S\).
A set \(S \subseteq V(G)\) is \emph{Carath\'eodory independent} if
\[
\Hull{S}
\setminus
\Union_{x \in S} \Hull{S \setminus \Set{x}}
\neq \emptyset,
\]
and it is \emph{exchange independent} if \(\card{S} = 1\) or there
exists \(p \in S\) such that
\[
\Hull{S \setminus \Set{p}}
\setminus
\Union_{x \in S \setminus \Set{p}}
\Hull{S \setminus \Set{x}}
\neq \emptyset.
\]
The \emph{Carath\'eodory} (resp., \emph{exchange}) \emph{number},
\(c_{\Tri}(G)\) (resp., \(e_{\Tri}(G)\)), is the size of a largest
Carath\'eodory (resp., exchange) independent subset of \(V(G)\).
It was shown by Anand--Anil--Changat--Narasimha-Shenoi--Ramla~\cite{AnandAnilEtAl2025b} in 2025 that
\(c_{\Tri}(G) \leq t(G)+1\) and \(e_{\Tri}(G) \leq t(G) + 2\), where \(t(G)\) is the number of triangles in \(G\), and that these bounds are tight.
They also computed \(c_{\Tri}(G)\) and \(e_{\Tri}(G)\) for a block graph \(G\) in terms of the number and arrangement of non-\(K_2\) blocks in \(G\).

In this paper, we point out a gap in the proof in~\cite{AnandAnilEtAl2025b} of the first inequality, \(c_{\Tri}(G) \leq t(G)+1\), which has consequences for the proof of the second inequality, \(e_{\Tri}(G) \leq t(G) + 2\), as well.
Moreover, the tightness results are inadvertently applied as characterizations of the extremal graphs, leading to incorrect computations of the Carath\'eodory and exchange numbers of block graphs in certain cases.
We fix these gaps by giving a full proof of the first inequality via a different route from that in~\cite{AnandAnilEtAl2025b}.
Together with the argument in~\cite{AnandAnilEtAl2025b}, this also completes the proof of the second inequality.
Our proof also leads to a characterization of the extremal graphs for each bound, which we use to compute the Carath\'eodory and exchange numbers of block graphs and to identify the extremal block graphs.
We also determine \(e_{\Tri}(G)\) exactly for \(k\)-trees for every \(k \geq 2\).

\keywords{Triangle convexity \and Carath\'eodory number
\and Exchange number \and Graph convexity \and Block graph \and \(k\)-tree.}
\end{abstract}

\section{Introduction}

\subsection{Abstract convexity spaces}
An \emph{abstract convexity space} on a finite set \(X\) is a collection \(\mathcal{C}\) of subsets of \(X\) such that \(\emptyset, X \in \mathcal{C}\) and if \(A,B \in \mathcal{C}\) then \(A \intersect B \in \mathcal{C}\).
The members of \(\mathcal{C}\) are called \emph{convex} sets, and the \emph{convex hull}, \(\Hull{S}\), of a set \(S \subseteq X\) is defined to be the smallest convex set containing \(S\).
Abstract convexity spaces were first defined in 1951 by Levi~\cite{Levi1951}.
The classical theorems of convex geometry in Euclidean spaces due to Carath\'eodory~\cite{Caratheodory1911}, Radon~\cite{Radon1921}, and Helly~\cite{Helly1923} (among others) give rise to important combinatorial parameters that measure the dimension of abstract convexity spaces.
Two such parameters that are central to this work are defined below.

In a convexity space \((X,\mathcal{C})\), a set \(S \subseteq X\) is said to be \emph{Carath\'eodory independent} if
\begin{equation}
	\Hull{S} \setminus \Union_{x \in S} \Hull{S \setminus \Set{x}} \neq \emptyset.
\end{equation}
The \emph{Carath\'eodory number} of \(X\), \(c(X)\), is the size of a largest Carath\'eodory independent subset of \(X\).
Similarly, a set \(S \subseteq X\) is said to be \emph{exchange independent} if \(\card{S} = 1\) or there exists \(p \in S\) such that
\begin{equation}
	\Hull{S \setminus \Set{p}} \setminus \Union_{x \in S \setminus \Set{p}} \Hull{S \setminus \Set{x}} \neq \emptyset.
\end{equation}
The \emph{exchange number} of \(X\), \(e(X)\), is the size of a largest exchange independent subset of \(X\).
These parameters are related by the following result of Sierksma:
\begin{theorem}[Sierksma~\cite{Sierksma1984}, 1984]\label{T:Sierksma}
	For any convexity space \((X,\mathcal{C})\), we have \(e(X) \leq c(X) + 1\).
\end{theorem}
We refer the reader to van de Vel~\cite{vanDeVel1993} for a detailed survey of abstract convexity structures.
We shall work instead with a concrete convexity structure, which we describe below.

\subsection{Graph convexities}
A \emph{graph convexity} is a convexity structure on the vertex set of a graph, usually assumed to be finite and simple.
One of the earliest papers on convexity in graphs was published in 1972 by Erd\H{o}s--Fried--Hajnal--Milner~\cite{ErdosFriedEtAl1972} and focussed on tournaments.
A systematic investigation of graph convexity was initiated in the 1980s by the works of Jamison~\cite{Jamison1981} and Harary~\cite{Harary1984}.
Many classical graph convexities are path-based, such as geodesic convexity~\cite{HararyNieminen1981}, monophonic convexity~\cite{Jamison1982}, \(P_3\) convexity~\cite{CentenoDouradoEtAl2009}, and so on (see~\cite{ChangatMathew1999,ChangatMulderEtAl2005}).
Graph convexities that are not path-based often arise from threshold functions~\cite{Chen2009} (to be defined!).
We refer the reader to the recent work of Ara\'ujo--Dourado--Protti--Sampaio~\cite{AraujoDouradoEtAl2025} for a detailed survey of these graph convexities, particularly from an algorithmic perspective.

\subsection{\(\Tri\)-convexity}
In this paper, we work with a notion of convexity which is neither path-based nor defined via threshold functions.
For a finite, simple graph \(G = (V,E)\), the \emph{\(\Tri\)-convexity} on \(G\) is the collection \(\mathcal{C}\) of those subsets of \(V(G)\) that are closed under formation of triangles; that is, \(C \in \mathcal{C}\) if and only if whenever \(x \in V(G)\) forms a triangle with two vertices in \(C\), we have \(x \in C\).\footnote{We make the stylistic choice to refer to this notion as \(\Tri\)-convexity, instead of \(\Delta\)-convexity as is done elsewhere in the literature, in order to emphasize that the defining structures of this graph convexity are triangles. It is also worth adding that \(\Tri\)-convexity is different from the notion of \emph{triangle path convexity} defined in~\cite{ChangatMathew1999}, which is instead a path-based convexity.}
The notion of \(\Tri\)-convexity was defined by Mulder~\cite{Mulder2008} in 2008 in the context of transit functions.
Several convexity parameters are known to be computationally hard in \(\Tri\)-convexity.
Anand--Anil--Changat--Dourado--Ramla~\cite{AnandAnilEtAl2020} in 2020 showed that determining whether the \emph{\(\Tri\)-hull number}, the minimum size of a set whose \(\Tri\)-convex hull is the entire vertex set, is at most \(k\) for a general graph is NP-complete.
Anand--Dourado--Narasimha-Shenoi--Ramla~\cite{AnandDouradoEtAl2022} in 2022 showed that computing the \emph{\(\Tri\)-convexity number}, the maximum size of a proper \(\Tri\)-convex set, is W[1]-hard and computing the \emph{\(\Tri\)-interval number}, the minimum size of a set \(S\) such that every vertex outside \(S\) forms a triangle with two vertices of \(S\), is NP-complete.
These papers also give polynomial-time algorithms for computing the corresponding parameters on certain special graph classes.
Bounds for the Helly and Radon numbers in \(\Tri\)-convexity were computed by Anand--Anil--Changat--Nair--Narasimha-Shenoi~\cite{AnandAnilEtAl2025a} in 2025, and bounds for the Carath\'eodory and exchange numbers were computed in Anand--Anil--Changat--Narasimha-Shenoi--Ramla~\cite{AnandAnilEtAl2025b} in 2025.
The results in the present paper concern the latter work.

\subsection{Known results}
We shall use the notations \(c_{\Tri}(G)\) and \(e_{\Tri}(G)\) to refer to the Carath\'eodory and exchange numbers of a graph \(G\) in \(\Tri\)-convexity.
Anand--Anil--Changat--Narasimha-Shenoi--Ramla~\cite{AnandAnilEtAl2025b} gave a bound for \(c_{\Tri}(G)\) in terms of \(t(G)\), the number of triangles in \(G\), as follows:
\begin{theorem}[{\cite[Theorem 2.3]{AnandAnilEtAl2025b}}]\label{T:Main}
	For a finite, simple graph \(G = (V,E)\), we have \(c_{\Tri}(G) \leq t(G) + 1\).
\end{theorem}

As a consequence of Theorems~\ref{T:Sierksma} and~\ref{T:Main}, the authors prove the following bound for \(e_{\Tri}(G)\):
\begin{corollary}[{\cite[Lemma 3.6]{AnandAnilEtAl2025b}}]\label{C:Main}
	For a finite, simple graph \(G = (V,E)\), we have \(e_{\Tri}(G) \leq t(G) + 2\).
\end{corollary}

The authors also show that these bounds are tight~\cite[Propositions 2.4, 3.8]{AnandAnilEtAl2025b}.
They further refine these results for block graphs by expressing the bounds in terms of the number and arrangement of non-\(K_2\) blocks.
\begin{theorem}[{\cite[Theorems 2.9 and 3.10]{AnandAnilEtAl2025b}}]\label{T:block}
Let $G$ be a block graph with $\ell$ blocks. Then the following statements hold:
\begin{enumerate}
    \item If $G$ contains no blocks isomorphic to the complete graph $K_2$ and all the blocks lie in a single chain, then
    \(
    c_{\Tri}(G) = \ell + 1 = e_{\Tri}(G).
    \)

    \item If $G$ contains no blocks isomorphic to the complete graph $K_2$ and the blocks lie in more than one chain, and $k$ is the number of blocks in the longest chain in $G$, then
    \(
    c_{\Tri}(G) = k + 1 = e_{\Tri}(G) - 1.
    \)

    \item If $G$ contains blocks isomorphic to the complete graph $K_2$ and $k$ is the maximum number of consecutive non-$K_2$ blocks in a chain, then
    \(
    c_{\Tri}(G) = k + 1 = e_{\Tri}(G) - 1.
    \)
\end{enumerate}
\end{theorem}

For \(2\)-connected chordal graphs, they obtain the following sharper bounds:
\begin{theorem}[{\cite[Theorem 2.8, Lemma 3.9]{AnandAnilEtAl2025b}}]\label{T:chordal}
	If \(G\) is a \(2\)-connected chordal graph, then \(c_{\Tri}(G) = 2\) and \(e_{\Tri}(G) \in \{ 2, 3\}\).
\end{theorem}

They further compute the Carath\'eodory and exchange numbers for several classes of graph products~\cite[\S4]{AnandAnilEtAl2025b}.

\subsection{Our work}

In Section~\ref{S:gaps}, we highlight a gap in the proof of Theorem~\ref{T:Main}.
To the best of our knowledge, this gap does not seem to have a straightforward fix.
Instead, we use a different route to give a full proof of Theorem~\ref{T:Main} in Sections~\ref{S:Preliminaries} and~\ref{S:proof-main}, and thereby also obtain a full proof of Corollary~\ref{C:Main} by applying Theorem~\ref{T:Sierksma}.

As a consequence of our proof, in Appendix~\ref{A:new12} we characterize the extremal graphs attaining equality in Theorem~\ref{T:Main} and in Corollary~\ref{C:Main} as follows:
\begin{theorem}\label{T:New1}
	Let \(G\) be a connected graph in which every edge lies on a triangle.
	Then, \(c_{\Tri}(G) = t(G) + 1\) if and only if \(G\) is obtained from a full binary tree \(\mathcal{T}\) by adding an edge between every pair of siblings.
\end{theorem}

\begin{theorem}\label{T:New2}
	Let \(G\) be a connected graph in which all but at most one edge lies on a triangle.
	Then, \(e_{\Tri}(G) = t(G) + 2\) if and only if \(G\) is obtained from a full binary tree \(\mathcal{T}\) by adding an edge between every pair of siblings, and then attaching a leaf to a vertex.
\end{theorem}

Theorems~\ref{T:New1} and~\ref{T:New2} strengthen the tightness results in~\cite[Propositions 2.4 and 3.8]{AnandAnilEtAl2025b} by identifying \emph{all} extremal graphs under the respective hypotheses.
Furthermore, the computations in Theorem~\ref{T:block} are partly incorrect, because the proofs inadvertently use the tightness results~\cite[Propositions 2.4 and 3.8]{AnandAnilEtAl2025b} as characterizations of the extremal graphs.
In Appendix~\ref{A:new345}, we use Theorems~\ref{T:New1} and~\ref{T:New2} to determine the Carath\'eodory and exchange numbers of block graphs exactly:

\begin{theorem}\label{T:New3}
For a block graph \(G\), let \(H_0\) denote a maximal connected block subgraph without any \(K_2\)-blocks, and let \(H_1\) denote a maximal connected block subgraph with exactly one \(K_2\)-block, which is also a pendant block of \(H_1\).
Then, \(c_{\Tri}(G) \leq \max_{H_0}\{b(H_0)\} + 1\) and \(e_{\Tri}(G) \leq \max_{H_0,H_1}\Set{b(H_0),b(H_1)} + 1\), where \(b(H_i)\) is the number of blocks of \(H_i\).
\end{theorem}

We also show in Appendix~\ref{A:new345} that our proof method yields the following characterizations of the extremal graphs attaining equality in the bounds of Theorem~\ref{T:New3}.
\begin{theorem}\label{T:New4}
Let \(G\) be a connected block graph without any \(K_2\)-blocks.
Then, \(c_{\Tri}(G) = b(G)+1\) if and only if \(G\) is obtained from a full binary tree with \(c_{\Tri}(G)\) many leaves by adding an edge between every pair of siblings, and then repeatedly adding vertices adjacent to selected cliques.	
\end{theorem}

\begin{theorem}\label{T:New5}
Let \(G\) be a connected block graph with exactly one \(K_2\)-block, and
suppose that this block is pendant.
Then, \(e_{\Tri}(G) = \max_{H_0,H_1}\{b(H_0),b(H_1)\}+1 = c_{\Tri}(G) + 1\) if and only if \(G\) is obtained from a full binary tree with \(c_{\Tri}(G)\) many leaves by adding an edge between every pair of siblings, and then repeatedly adding vertices adjacent to selected cliques, and then attaching a leaf to one vertex.
\end{theorem}

Lastly, if \(G\) is a \(k\)-tree with \(k \geq 2\), then \(G\) is a \(2\)-connected chordal graph.
Hence, by Theorem~\ref{T:chordal}, \(c_{\Tri}(G) = 2\) and \(e_{\Tri}(G) \in \Set{2, 3}\).
In Appendix~\ref{A:new6}, we characterize exactly the \(k\)-trees, for \(k \geq 2\), for which \(e_{\Tri}(G) = 2\).
\begin{theorem}\label{T:k-tree}
	Let \(G\) be a \(k\)-tree, where \(k \geq 2\).
	Then, \(e_{\Tri}(G) = 2\) if \(G\) is a split graph, and \(e_{\Tri}(G) = 3\) otherwise.
\end{theorem} 

\section{The gaps in Theorems~\ref{T:Main} and~\ref{T:block}}\label{S:gaps}

The outline of the proof of Theorem~\ref{T:Main} (Theorem 2.3 in \cite{AnandAnilEtAl2025b}) goes as follows.
Suppose that \(t(G) \geq 2\), since the result is easily seen to be true for \(t(G) = 0\) and \(1\).
Suppose that \(S\) is a Carath\'eodory independent set in \(G\).
It is easy to see that every vertex of \(S\) belongs to some triangle and \(S\) cannot contain any triangle, else it would be Carath\'eodory dependent.
If \(\card{S} \geq t(G) + 2\), then by the pigeonhole principle there are at least two triangles that contain exactly two vertices of \(S\).
Now, the authors argue that if \(T_1,\dotsc,T_r\) (\(r \geq 2\)) are all such triangles in \(G\), say with \(T_i\) containing the vertices \(u_i,v_i \in S\), then
\begin{equation}\label{Eq:problem}
	\Hull{S} = \Union_{i = 1}^r \bigl( \Hull{S \setminus \Set{u_i}} \union \Hull{S \setminus \Set{v_i}}\bigr),
\end{equation}
which contradicts that \(S\) is Carath\'eodory independent.
Hence, \(\card{S} \leq t(G) + 1\), as required.

We found that \eqref{Eq:problem} is not easily justified, and a further argument is necessary to close this gap.
Our proof follows a different approach that does not require one to establish \eqref{Eq:problem}.
It has the added advantage that it allows us to precisely determine the graphs that attain equality in Theorem~\ref{T:Main}.

Next, the tightness result for the Carath\'eodory number (Proposition 2.4 in \cite{AnandAnilEtAl2025b}) shows that a chain of triangles attains the equality; for example, Figure~\ref{fig:triangle-chain-four} shows a graph for which \(c_{\Tri}(G) = 5 = t(G) + 1\).
However, it is easy to see that the blocks (in this case, triangles) need not lie in a chain for equality to hold: Figure~\ref{fig:block-counterexample} shows a graph in which \(c_{\Tri}(G) = 5 = t(G) + 1\), but the triangles do not all lie on a chain.
Similar counterexamples to some of the cases of Theorem~\ref{T:block} can be constructed for the exchange number as well.

\section{Preliminaries}\label{S:Preliminaries}

For a set \(S \subseteq V(G)\), the \emph{\(\Tri\)-interval} of \(S\), \(I_{\Tri}(S)\), is defined as
\[
I_{\Tri}(S) \defn S \cup \Set{x \in V(G) : \exists\, y, z \in S \text{ such that } \{x,y,z\} \text{ is a triangle of } G}.
\]
Setting \(I_{\Tri}^0(S) \defn S\) and \(I_{\Tri}^{n+1}(S) \defn I_{\Tri}(I_{\Tri}^n(S))\), we have that \(\Hull{S}\) is the eventual stabilisation \(\bigcup_{n \ge 0} I_{\Tri}^{n}(S)\), which is reached in finitely many steps since \(G\) is a finite graph.
The \emph{iteration time} of \(S\) is the minimum \(n\) such that \(I_{\Tri}^{n}(S) = \Hull{S}\).

We recall two basic structural properties of Carath\'eodory independent sets:

\begin{lemma}\label{L:basic}
Let \(S \subseteq V(G)\) be a Carath\'eodory independent set. Then:
\begin{enumerate}
\item\label{L:basic1} If \(\card{S} \geq 2\), then every vertex of \(S\) lies in a triangle of \(G\).
\item\label{L:basic2} No triangle of \(G\) is contained in \(S\).
\end{enumerate}
\end{lemma}

\begin{proof}
If \(u \in S\) does not lie in any triangle, then \(\Hull{S} = \Hull{S \setminus \Set{u}} \cup \Set{u}\),
so \(S\) is not Carath\'eodory independent, a contradiction.

If \(u,v,w \in S\) form a triangle, then since \(v, w \in S \setminus \Set{u}\) and \(uvw\) is a triangle, we have \(u \in \Itr{S \setminus \Set{u}} \subseteq \Hull{S \setminus \Set{u}}\), so \(\Hull{S} = \Hull{S \setminus \Set{u}}\).
By symmetry, \(\Hull{S} = \Hull{S \setminus \Set{v}} = \Hull{S \setminus \Set{w}}\) as well, contradicting that \(S\) is Carath\'eodory independent.
\end{proof}

Let \(S \subseteq V(G)\) be a Carath\'eodory independent set, and consider the induced subgraph \(G[S]\).
For each connected component \(H\) of \(G[S]\), define
\[
\mathcal{T}_H \defn \Set{T : T \text{ is a triangle of } G \text{ with } V(T)\cap V(H)\neq \emptyset}.
\]

\begin{lemma}\label{L:partition}
If \(H\) and \(H'\) are distinct connected components of \(G[S]\), then
\(\mathcal{T}_H \cap \mathcal{T}_{H'} = \emptyset\).
\end{lemma}

\begin{proof}
Let \(T\) be a triangle of \(G\) that meets \(S\).
By Lemma~\ref{L:basic}(\ref{L:basic2}), \(\card{V(T)\cap S} \in \{1,2\}\).
If \(\card{V(T)\cap S}=2\), those two vertices are adjacent and lie in the same component of \(G[S]\).
If \(\card{V(T)\cap S}=1\), the vertex lies in a unique component.
Thus \(T\) belongs to \(\mathcal{T}_H\) for exactly one component \(H\).
\end{proof}

\begin{lemma}\label{L:counting}
Let \(H\) be a connected component of \(G[S]\), and suppose every edge of \(G\) lies in a triangle.
\begin{enumerate}
\item If \(H\) has no edges, then \(\card{\mathcal{T}_H} \geq \card{V(H)}\).
\item If \(\card{E(H)} \geq \card{V(H)}\), then \(\card{\mathcal{T}_H} \geq \card{V(H)}\).
\item If \(H\) is a tree with at least one edge, then \(\card{\mathcal{T}_H} \geq \card{V(H)} - 1\).
Equality holds if and only if every edge of \(H\) is the base of exactly one triangle, and no triangle meets \(H\) in exactly one vertex.
\end{enumerate}
\end{lemma}

\begin{proof}
Since every edge of \(G\) lies in a triangle, each edge of \(H\) contributes at least one triangle to \(\mathcal{T}_H\).
If \(H\) has no edges, then \(H\) is an isolated vertex \(u\), and \(u\) lies in a triangle by Lemma~\ref{L:basic}(\ref{L:basic1}), giving \(\card{\mathcal{T}_H} \geq 1 = \card{V(H)}\).
If \(\card{E(H)} \geq \card{V(H)}\), then \(\card{\mathcal{T}_H} \geq \card{E(H)} \geq \card{V(H)}\).
If \(H\) is a tree with at least one edge, then \(\card{E(H)} = \card{V(H)} - 1\), so \(\card{\mathcal{T}_H} \geq \card{V(H)} - 1\), with equality if and only if no edge of \(H\) is the base of more than one triangle and no triangle meets \(H\) in exactly one vertex.
\end{proof}

A connected component \(H\) of \(G[S]\) is called \emph{exceptional} if it is a tree with at least one edge, every edge of \(H\) is the base of exactly one triangle in \(G\), and no triangle of \(G\) meets \(H\) in exactly one vertex.
Exceptional components are precisely those tree components for which \(\card{\mathcal{T}_H} = \card{V(H)} - 1\).

\section{Proof of Theorem~\ref{T:Main}}\label{S:proof-main}

It suffices to assume that \(G\) is connected. For \(t(G)\ge 1\), deleting edges that lie on no triangle does not change the \(\Tri\)-hull operation, and deleting vertices that lie on no triangle does not affect the maximum size of a Carath\'eodory independent set: by Lemma~\ref{L:basic}, no Carath\'eodory independent set of size at least two contains such a vertex, while the presence of a triangle gives a Carath\'eodory independent set of size two. Thus, in the nontrivial cases, we may assume that every edge and vertex lies on a triangle.

The proof is by induction on \(t = t(G)\).
If \(t=0\), then \(G\) is triangle-free, so \(\Hull{S} = S\) for every \(S \subseteq V(G)\), giving \(c_\triangle(G) = 1 = t+1\).

If \(t=1\), let \(\{u,v,w\}\) be the unique triangle.
By Lemma~\ref{L:basic}, any Carath\'eodory independent set of size at least two is contained in \(\{u,v,w\}\) and does not equal \(\{u,v,w\}\).
Any two-element subset of \(\{u,v,w\}\) is Carath\'eodory independent, so \(c_\triangle(G) = 2 = t+1\).

Now let \(t \geq 2\) and suppose the result holds for all graphs with fewer than \(t\) triangles.
Let \(S\) be a Carath\'eodory independent set in \(G\).
Let \(H_1,\dots,H_m\) be the connected components of \(G[S]\), and let \(\mathcal{E}\) denote the collection of exceptional components.
By Lemmas~\ref{L:partition} and~\ref{L:counting}, the families \(\mathcal{T}_{H_1},\ldots,\mathcal{T}_{H_m}\) are pairwise disjoint, and
\[
\card{S} = \sum_{i=1}^m \card{V(H_i)} \leq \sum_{i=1}^m \card{\mathcal{T}_{H_i}} + \card{\mathcal{E}} \leq t + \card{\mathcal{E}}.
\]
If \(\card{\mathcal{E}} \leq 1\), then \(\card{S} \leq t + \card{\mathcal{E}} \leq t + 1\) and we are done.
Henceforth assume \(\card{\mathcal{E}} \geq 2\).

\medskip
\noindent\textbf{Case 1: Inductive reduction.}
Suppose there exists \(H \in \mathcal{E}\) and a leaf \(u\) of \(H\) with the following property: letting \(v\) be the unique neighbour of \(u\) in \(H\) and \(x\) the unique apex of the triangle on the edge \(uv\), the edge \(xv\) does not lie on any \(4\)-cycle in \(G - u\).

Form the graph \(G'\) by deleting \(u\) and contracting the edge \(xv\) to a new vertex \(\bar{v}\), and let \(S' \defn (S \setminus \Set{u,v}) \cup \Set{\bar{v}}\).

\begin{lemma}\label{L:CI-preservation}
\(S'\) is Carath\'eodory independent in \(G'\).
\end{lemma}

\begin{proof}
Let \(\phi \colon V(G) \setminus \{u\} \to V(G')\) send both \(x\) and \(v\) to \(\bar{v}\) and fix all other vertices.
Since $H$ is exceptional and $u$ is a leaf of $H$, $\triangle uvx$ is the only triangle containing $u$. By the hypothesis of Case~1, contracting $xv$ in $G-u$ creates no new triangle.

Choose
\[
q\in \operatorname{Hull}_G(S)\setminus\bigcup_{p\in S}\operatorname{Hull}_G(S\setminus\{p\})
\]
witnessing the Carath\'eodory independence of $S$. Since Case~1 is reached only after $|\mathcal E|\ge2$, we have $|S|\ge3$. Hence $q\notin\{u,v,x\}$: the vertices $u$ and $v$ belong to $S\setminus\{p\}$ for some $p\ne u,v$, while for any $p\in S\setminus\{u,v\}$ the vertices $u,v\in S\setminus\{p\}$ generate $x$.

The closure sequence generating $q$ from $S$ maps to a closure sequence generating $\phi(q)$ from $S'$: the only use of $u$ can be through $\triangle uvx$, and in $G'$ the roles of $v$ and $x$ are represented by the initial vertex $\bar v$.

Now fix $p'\in S'$. If $p'\ne\bar v$, put $p=p'$. Then $u,v\in S\setminus\{p\}$, so $x$ enters the first iterate of $S\setminus\{p\}$. Because the contraction creates no new triangles, every closure step in $G'$ from $S'\setminus\{p'\}$ lifts to a valid closure step in $G$ from $S\setminus\{p\}$. Since $q\notin\{x,v\}$, it follows that
\[
\phi(q)\notin \operatorname{Hull}_{G'}(S'\setminus\{p'\}).
\]

It remains to take $p'=\bar v$, corresponding to deleting $v$ in $G$. We claim that $\bar v$ never enters $\operatorname{Hull}_{G'}(S'\setminus\{\bar v\})$. Otherwise, let $\triangle \bar v ab$ be the triangle that first generates $\bar v$. Before this step no closure step uses $\bar v$, so the vertices $a,b$ lift to vertices in $\operatorname{Hull}_G(S\setminus\{v\})$. Since the contraction creates no new triangle, $\triangle \bar v ab$ lifts to either $\triangle vab$ or $\triangle xab$ in $G-u$. In the first case $v\in\operatorname{Hull}_G(S\setminus\{v\})$. In the second case $x\in\operatorname{Hull}_G(S\setminus\{v\})$, and then $u,x$ generate $v$ through $\triangle uvx$. Both contradict the Carath\'eodory independence of $S$. Thus $\bar v$ is never generated, and every closure step from $S'\setminus\{\bar v\}$ lifts directly to $G$ from $S\setminus\{v\}$. Hence again
\[
\phi(q)\notin \operatorname{Hull}_{G'}(S'\setminus\{\bar v\}).
\]

Therefore
\[
\phi(q)\in \operatorname{Hull}_{G'}(S')\setminus\bigcup_{p'\in S'}\operatorname{Hull}_{G'}(S'\setminus\{p'\}),
\]
so $S'$ is Carath\'eodory independent in $G'$.
\end{proof}
Since $u$ is a leaf of the exceptional component $H$, $\triangle uvx$ is the only triangle containing $u$. Moreover, $xv$ lies in no triangle of $G-u$. Indeed, if $\triangle xvr$ were such a triangle and $r\notin S$, it would meet $H$ in the single vertex $v$, contradicting the exceptionality of $H$. If $r\in S$, then $r\in V(H)$ and, in $\operatorname{Hull}(S\setminus\{u\})$, the vertices $v,r$ generate $x$, after which $v,x$ regenerate $u$, again a contradiction.

Finally, contracting $xv$ creates no new triangle, because any genuinely new triangle at $\bar v$ would correspond to a $4$-cycle containing the edge $xv$ in $G-u$. Nor can two distinct surviving triangles be identified by the contraction, since that would give a triangle containing $xv$ in $G-u$. Hence
\(
t(G')=t(G)-1.
\)
By the induction hypothesis, \(\card{S'} \leq t(G') + 1 = t(G)\), and therefore \(\card{S} = \card{S'} + 1 \leq t + 1\).

\medskip
\noindent\textbf{Case 2: The \(4\)-cycle configuration.}
Suppose Case~1 does not apply. Then, for every $H\in\mathcal E$
and every leaf $u$ of $H$, if $v$ is the unique neighbor of $u$
in $H$ and $x$ is the unique apex of the triangle on the edge $uv$,
the edge $xv$ lies on a $4$-cycle in $G-u$.

Fix $H\in\mathcal E$ and a leaf $u$ of $H$, and write such a
$4$-cycle as
\[
x-v-s-z-x.
\]

First suppose that $s\in S$. Since $vs\in E(G[S])$ and $v\in V(H)$,
we have $s\in V(H)$. We claim that $z\notin S$. Otherwise
$z\in V(H)$. Choose a triangle $\triangle xza$ containing the edge
$xz$. Since $H$ is exceptional and $x\notin S$, we must have
$a\in V(H)$. Moreover, $a\neq u$, since $u$ is a leaf of $H$.
Thus $z,a\in S\setminus\{u\}$ generate $x$, after which $v,x$
generate $u$, contradicting the Carath\'eodory independence of $S$.

Hence $z\notin S$. Choose a triangle $\triangle szr$ containing the
edge $sz$. Since $s\in V(H)$ and $H$ is exceptional, we have
$r\in V(H)$. Thus $z$ is the unique apex of the exceptional edge
$sr$.

Now choose a triangle
\[
T=\triangle xza
\]
containing the edge $xz$. We claim that $a\notin S$. If
$a\in V(H)$, then $T$ meets $H$ in exactly one vertex, contradicting
the exceptionality of $H$. If $a\in S\setminus V(H)$, then in
$\operatorname{Hull}(S\setminus\{u\})$ the vertices $s,r$ generate
$z$, then $z,a$ generate $x$, and finally $v,x$ generate $u$, again
a contradiction. Hence $T$ is disjoint from $S$.

We defer this configuration for the moment and consider the remaining
case $s\notin S$. Since the edge $vs$ lies on a triangle and $H$ is
exceptional, there exists $w\in V(H)$ such that $\triangle vsw$ is
a triangle. Set $y:=s$. Then $y$ is the unique apex of the
exceptional edge $vw$, and the $4$-cycle has the form
\[
x-v-y-z-x.
\]
Choose triangles
\[
A=\triangle xza
\qquad\text{and}\qquad
B=\triangle yzb
\]
containing the edges $xz$ and $yz$, respectively.

We establish the structural constraints needed for the case analysis.

\begin{lemma}\label{L:atmost-one}
At most one of \(a\), \(z\), \(b\) belongs to \(S\).
\end{lemma}

\begin{proof}
We show that no two of \(a\), \(z\), \(b\) can simultaneously belong to \(S\).

Suppose \(a, b \in S\) (with \(a \neq b\)).
Since \(u, v \in S\), triangle \(\triangle uvx\) gives \(x \in \Itr{S \setminus \Set{b}}\).
Then \(x, a \in \Hull{S \setminus \Set{b}}\), so \(\triangle xza\) gives \(z \in \Hull{S \setminus \Set{b}}\).
Then \(z, y \in \Hull{S \setminus \Set{b}}\), so \(\triangle yzb\) gives \(b \in \Hull{S \setminus \Set{b}}\), a contradiction.

Suppose \(a, z \in S\).
Then \(a, z \in S \setminus \{v\}\), so \(\triangle xza\) gives \(x \in \Itr{S \setminus \Set{v}}\).
Then \(u \in S \setminus \{v\}\) and \(\triangle uvx\) give \(v \in \Hull{S \setminus \Set{v}}\), a contradiction.

Suppose \(b, z \in S\).
Then \(b, z \in S \setminus \{v\}\), so \(\triangle yzb\) gives \(y \in \Itr{S \setminus \Set{v}}\).
Then \(w \in S \setminus \{v\}\) and \(\triangle vwy\) give \(v \in \Hull{S \setminus \Set{v}}\), a contradiction.
\end{proof}

\begin{lemma}\label{L:outside-exceptional}
In the \(4\)-cycle configuration:
\begin{enumerate}
\item\label{L:outside-exceptional1} \(x, y \notin S\);
\item\label{L:outside-exceptional2} no element of \(\{a, z, b\} \cap S\) is a vertex of any component in \(\mathcal{E}\).
\end{enumerate}
\end{lemma}

\begin{proof}
For~(1): since \(u, v \in S\) and \(\{u,v,x\}\) is a triangle, Lemma~\ref{L:basic}(\ref{L:basic2}) gives \(x \notin S\).
The same argument with \(v, w \in S\) and \(\triangle vwy\) gives \(y \notin S\).

For~(2): let \(q \in \{a,z,b\} \cap S\).
By Lemma~\ref{L:outside-exceptional}(\ref{L:outside-exceptional1}) and Lemma~\ref{L:atmost-one}, the two vertices of \(\triangle xza\) or \(\triangle yzb\) other than \(q\) all lie outside \(S\): if \(q = a\) then \(x, z \notin S\); if \(q = z\) then \(x, a \notin S\) and \(y, b \notin S\); if \(q = b\) then \(y, z \notin S\).
In each case the relevant triangle meets \(S\) in exactly the single vertex \(\{q\}\).
If \(q\) belonged to some exceptional component \(H' \in \mathcal{E}\), this triangle would meet \(H'\) in exactly one vertex, contradicting the exceptionality of \(H'\).
\end{proof}

\begin{lemma}\label{L:unique-apex}
Each vertex of \(G\) is the apex of at most one exceptional edge across all components in \(\mathcal{E}\).
\end{lemma}

\begin{proof}
Suppose that a vertex \(x\) is the apex of two distinct exceptional edges \(uv\) and \(u'v'\), possibly belonging to the same exceptional component.
Since the two edges are distinct, one endpoint of the first edge, say \(u\), is not an endpoint of the second.
In \(\Hull(S\setminus\{u\})\), both endpoints \(u',v'\) are present, so they generate \(x\).
The other endpoint \(v\) of \(uv\) is also present, and hence \(v,x\) regenerate \(u\), contradicting the Carath\'eodory independence of \(S\).
Thus each vertex is the apex of at most one exceptional edge.
\end{proof}

\begin{lemma}\label{L:atmost-two}
A triangle \(T\) of \(G\) can serve as a spare triangle for at most two components in \(\mathcal{E}\).
\end{lemma}

\begin{proof}
For \(T\) to serve as a spare for \(H \in \mathcal{E}\), it arises as either \(\triangle xza\) (with \(x\) the apex of exceptional edge \(uv\) in \(H\)) or \(\triangle yzb\) (with \(y\) the apex of exceptional edge \(vw\) in \(H\)).
In either case, exactly one vertex of \(T\) is the apex of an exceptional edge of \(H\).
By Lemma~\ref{L:unique-apex}, each vertex is the apex of at most one exceptional edge, so \(T\) can serve at most three components in \(\mathcal{E}\) a priori.

We show that all three vertices of \(T\) cannot simultaneously be apices of exceptional edges.
Suppose the vertex set of \(T\) is \(\{x_1, x_2, x_3\}\), where each \(x_i\) is the apex of exceptional edge \(u_i v_i\) in \(H_i \in \mathcal{E}\).
In \(\Hull{S \setminus \Set{u_1}}\), the vertices \(u_2, v_2\) and \(u_3, v_3\) all lie in \(S \setminus \{u_1\}\), so \(x_2\) and \(x_3\) both enter the first iterate.
Since \(T\) is a triangle and \(x_2, x_3\) are in the hull, \(x_1\) enters the second iterate.
Then \(v_1 \in S \setminus \{u_1\}\) and \(\triangle u_1 v_1 x_1\) give \(u_1 \in \Hull{S \setminus \Set{u_1}}\), contradicting that \(S\) is Carath\'eodory independent.
Hence at most two vertices of \(T\) are apices of exceptional edges, and \(T\) can serve at most two components in \(\mathcal{E}\).
\end{proof}

We first dispose of the configuration above in which $s\in S$.
Recall that in this case we obtained a triangle
\[
T=\triangle xza
\]
disjoint from $S$, where $x$ and $z$ are the unique apices of
exceptional edges $uv$ and $sr$, respectively, in the same component
$H$.

We claim that $T$ cannot serve as a spare triangle for any other
exceptional component. By Lemma~7, neither $x$ nor $z$ can be the
apex of an exceptional edge in another component. Suppose that $a$
is the apex of an exceptional edge $pq$ in some $H'\neq H$. In
$\operatorname{Hull}(S\setminus\{p\})$, the vertices $u,v$ generate
$x$, while $s,r$ generate $z$. Hence $x,z$ generate $a$, and then
$q,a$ generate $p$, contradicting the Carath\'eodory independence
of $S$. Thus $T$ is available exclusively for $H$, and we assign it
to $H$.

We now return to the case $s\notin S$, with $A=\triangle xza$ and $B=\triangle yzb$ as defined above.

\medskip
\noindent\textit{Subcase \textup{(a)}: \(b \in S\) and \(a \neq b\).}
By Lemma~\ref{L:atmost-one}, we have \(z \notin S\) and \(a \notin S\).
Hence no vertex of \(\triangle xza\) belongs to \(S\), so \(\triangle xza \notin \mathcal{T}_{H_i}\) for any component \(H_i\).
We show that \(\triangle xza\) cannot serve as a spare for any other \(H' \in \mathcal{E}\), and is thus fully available for \(H\).

For \(\triangle xza\) to serve \(H'\), one of \(x, z, a\) must be the apex of an exceptional edge in \(H'\).
By Lemma~\ref{L:unique-apex}, it cannot be \(x\).
By Lemma~\ref{L:outside-exceptional}(\ref{L:outside-exceptional2}), since \(b \in S\), the vertex \(b\) does not belong to any component in \(\mathcal{E}\).
Therefore \(b \neq u'\) for any leaf \(u'\) of any \(H' \in \mathcal{E}\), and so \(b \in S \setminus \{u'\}\) for any such leaf.

If \(z\) is the apex of exceptional edge \(u'v'\) in \(H'\), consider \(\Hull{S \setminus \Set{u'}}\).
Since \(v, w \in S \setminus \{u'\}\), triangle \(\triangle vwy\) gives \(y\) in the first iterate.
Since \(y\) is in the first iterate and \(b \in S \setminus \{u'\}\), triangle \(\triangle yzb\) gives \(z\) in the second iterate.
Then \(v' \in S \setminus \{u'\}\) and \(\triangle u'v'z\) give \(u' \in \Hull{S \setminus \Set{u'}}\), a contradiction.

If \(a\) is the apex of exceptional edge \(u'v'\) in \(H'\), consider \(\Hull{S \setminus \Set{u'}}\).
Since \(u, v \in S \setminus \{u'\}\), triangle \(\triangle uvx\) gives \(x\) in the first iterate.
Since \(v, w \in S \setminus \{u'\}\), triangle \(\triangle vwy\) gives \(y\) in the first iterate.
Since \(y\) and \(b \in S \setminus \{u'\}\) are in the hull, \(\triangle yzb\) gives \(z\) in the second iterate.
Then \(z\) and \(x\) are in the hull, so \(\triangle xza\) gives \(a\) in the third iterate.
Then \(v' \in S \setminus \{u'\}\) and \(\triangle u'v'a\) give \(u' \in \Hull{S \setminus \Set{u'}}\), a contradiction.

Hence \(\triangle xza\) is fully available and is assigned to \(H\).

\medskip
\noindent\textit{Subcase \textup{(b)}: \(a \in S\) and \(a \neq b\).}
By the same argument with the roles of \(a, b\) and \(x, y\) interchanged, the triangle \(\triangle yzb\) is fully available and is assigned to \(H\).

\medskip
\noindent\textit{Subcase \textup{(c)}: \(a = b = q \in S\).}
By Lemma~\ref{L:atmost-one}, $z\notin S$. Let $C$ be the component of $G[S]$ containing $q$. By Lemma~\ref{L:outside-exceptional}(\ref{L:outside-exceptional2}), $C$ is non-exceptional. Both $A=\triangle xzq$ and $B=\triangle yzq$ belong to $\mathcal T_C$, and they are distinct because $x\ne y$.

Neither $A$ nor $B$ can serve a different exceptional component. For example, $x$ is excluded by Lemma~\ref{L:unique-apex} and $q\in S$ cannot be the apex of an exceptional edge. If $z$ were the apex of an exceptional edge $pr$ in a component $H'\ne H$, then in $\operatorname{Hull}(S\setminus\{p\})$ the vertices $u,v$ generate $x$, the vertices $v,w$ generate $y$, the vertices $y,q$ generate $z$ through $B$, and then $r,z$ regenerate $p$, a contradiction. The same argument applies to $B$.
Thus these two triangles contribute two additional members of \(\mathcal T_C\). We associate one of these with \(H\); the counting argument below shows that this can be done for all exceptional components simultaneously.

\medskip
\noindent\textit{Subcase \textup{(d)}: none of \(a\), \(b\), \(z\) belongs to \(S\). Then both $A$ and $B$ are disjoint from $S$.}

If $A=B$, then necessarily $A=B=\triangle xyz$. This common triangle cannot serve as a spare triangle for another exceptional component. Indeed, $x$ and $y$ are already apices of exceptional edges of $H$. If $z$ were the apex of an exceptional edge $pr$ in a component $H'\ne H$, then in $\Hull{S\setminus\{p\}}$ the edges $uv$ and $vw$ generate $x$ and $y$, the triangle $\triangle xyz$ generates $z$, and then $r,z$ regenerate $p$, a contradiction. Hence the common triangle is assigned weight $1$ to $H$.

If $A\ne B$, then by Lemma~\ref{L:atmost-two} each of $A$ and $B$ can serve at most two exceptional components. Assign weight $1/2$ from each to $H$. Thus $H$ receives total weight $1$, and no triangle used in this manner receives total weight greater than $1$.

\medskip
\noindent\textit{Subcase \textup{(e)}: \(z \in S\).}
By Lemma~\ref{L:atmost-one}, $a,b\notin S$. Let $C$ be the component of $G[S]$ containing $z$. By Lemma~\ref{L:outside-exceptional}(\ref{L:outside-exceptional2}), $C$ is non-exceptional, and both $A$ and $B$ belong to $\mathcal T_C$.

First note that $A\ne B$. Otherwise $A=B=\triangle xyz$; but then in $\Hull{S\setminus\{z\}}$ the vertices $u,v$ generate $x$ and the vertices $v,w$ generate $y$, after which $x,y$ regenerate $z$, contradicting the Carath\'eodory independence of $S$.

Neither $A$ nor $B$ can serve as a spare triangle for another exceptional component. For $A$, the vertex $x$ is excluded by Lemma~\ref{L:unique-apex}, and $z\in S$ cannot be the apex of an exceptional edge. If $a$ were the apex of an exceptional edge $pr$ in a component $H'\ne H$, then in $\operatorname{Hull}(S\setminus\{p\})$ the vertices $u,v$ generate $x$, while $z$ is already present; hence $x,z$ generate $a$, and then $r,a$ regenerate $p$, a contradiction. The argument for $B$ is symmetric.

If $C$ is an isolated vertex, then $|\mathcal T_C|\ge2=|V(C)|+1$. If $C$ contains a cycle, the edge-triangles of $C$ already give at least $|V(C)|$ triangles, and $A,B$ are additional. If $C$ is a non-exceptional tree, its edge-triangles give at least $|V(C)|-1$ triangles, and the additional triangles $A,B$ give $|\mathcal T_C|\ge|V(C)|+1$. Thus in every case \(C\) contains enough additional triangles to assign one to \(H\). We make this assignment here; the counting argument below shows that these assignments can be made simultaneously when the same component \(C\) is associated with several exceptional components.

\medskip
For completeness, suppose the same non-exceptional component $C$ is used by $m$ exceptional components through either the subcase $a=b\in S$ or the subcase $z\in S$. Each such assignment contributes two distinct triangles of $\mathcal T_C$ meeting $C$ in exactly one vertex. The exclusivity arguments above imply that these $2m$ triangles are distinct across the $m$ exceptional components.

If $C$ is an isolated vertex, then
\[
|\mathcal T_C|-|V(C)|\ge 2m-1\ge m.
\]
If $C$ is a tree, its edge-triangles contribute at least $|V(C)|-1$ triangles, and hence
\[
|\mathcal T_C|-|V(C)|\ge 2m-1\ge m.
\]
If $C$ contains a cycle, its edge-triangles already contribute at least $|V(C)|$ triangles, so
\[
|\mathcal T_C|-|V(C)|\ge 2m\ge m.
\]
Thus $C$ has enough surplus to supply one full unit to each of the $m$ exceptional components assigned to it.

\medskip
Let $\mathcal F$ be the set of triangles of $G$ disjoint from $S$, and for each non-exceptional component $C$ of $G[S]$ set
\[
\delta(C):=|\mathcal T_C|-|V(C)|\ge0.
\]
The preceding assignments give one unit to every exceptional component, while no triangle of $\mathcal F$ is assigned total weight greater than $1$, and no non-exceptional component contributes more than its surplus $\delta(C)$. Hence
\[
|\mathcal E|\le |\mathcal F|+\sum_{C\notin\mathcal E}\delta(C).
\]
Since the families $\mathcal T_C$ are pairwise disjoint and every triangle either belongs to exactly one of them or lies in $\mathcal F$, we obtain
\[
\begin{aligned}
|S|
&=\sum_{H\in\mathcal E}\bigl(|\mathcal T_H|+1\bigr)
 +\sum_{C\notin\mathcal E}|V(C)|\\
&=\sum_C|\mathcal T_C|+|\mathcal E|
 -\sum_{C\notin\mathcal E}\delta(C)\\
&\le \sum_C|\mathcal T_C|+|\mathcal F|\\
&=t(G).
\end{aligned}
\]
Hence $|S|\le t(G)\le t(G)+1$, completing Case~2 and the induction.

\section{Concluding remarks}

We have proved Theorem~\ref{T:Main} by a component-counting argument and carefully tracking the triangles that are attached to each component.
By applying Theorem~\ref{T:Sierksma}, this also completed the proof of Corollary~\ref{C:Main}.
In the appendices, we give extremal characterizations of Theorem~\ref{T:Main} and Corollary~\ref{C:Main}, and apply the results to block graphs and \(k\)-trees, to prove Theorems~\ref{T:New1}--\ref{T:k-tree}.

We remark that though our proof is hands-on, and proves the inequalities by counting triangles in \(G\) and associating them with the vertices in a largest Carath\'eodory independent set, the mapping is case-dependent and requires careful bookkeeping at several places.
The underlying structure is clearest in the extremal case, but we are unable to reduce the complexity of the proof in the general case.
We believe a more direct proof of Theorem~\ref{T:Main} could shed further light on the structure of \(\Tri\)-convexity.

\begin{credits}
\subsubsection{\ackname} The authors thank Narayanan N.\ for suggesting the computation of the exchange number of \(k\)-trees at the Annual Meru Combinatorics Conference 2026, and Manoj Changat for helpful discussions at ADMA-ICDM 2026.

\subsubsection{\discintname}
The authors have no competing interests to declare that are relevant to the content of this article.
\end{credits}

\bibliographystyle{splncs04}
\bibliography{references}

@article{AnandAnilEtAl2020,
  author  = {Anand, B. S. and Anil, A. and Changat, M. and Dourado, M. C. and Ramla, S. S.},
  title   = {Computing the hull number in {$\Delta$}-convexity},
  journal = {Theoretical Computer Science},
  volume  = {844},
  year    = {2020},
  pages   = {217--226}
}

@inproceedings{AnandAnilEtAl2025a,
  author    = {Anand, B. S. and Anil, A. and Changat, M. and Nair, R. S. and Narasimha-Shenoi, P. G.},
  title     = {Helly number, Radon number and rank in {$\Delta$}-convexity on graphs},
  booktitle = {Algorithms and Discrete Applied Mathematics. 11th International Conference, CALDAM 2025, Coimbatore, India, February 13--15, 2025, Proceedings},
  editor    = {Gaur, D. and others},
  series    = {Lecture Notes in Computer Science},
  volume    = {15536},
  publisher = {Springer},
  address   = {Cham},
  year      = {2025},
  pages     = {292--306}
}

@article{AnandAnilEtAl2025b,
  author  = {Anand, B. S. and Anil, A. and Changat, M. and Narasimha-Shenoi, P. G. and Ramla, S. S.},
  title   = {Carath{\'e}odory number and exchange number in {$\Delta$}-convexity},
  journal = {Journal of Combinatorial Mathematics and Combinatorial Computing},
  volume  = {126},
  year    = {2025},
  pages   = {11--27}
}

@article{AnandDouradoEtAl2022,
  author  = {Anand, B. S. and Dourado, M. C. and Narasimha-Shenoi, P. G. and Ramla, S. S.},
  title   = {On the {$\Delta$}-interval and the {$\Delta$}-convexity numbers of graphs and graph products},
  journal = {Discrete Applied Mathematics},
  volume  = {319},
  year    = {2022},
  pages   = {487--498}
}

@book{AraujoDouradoEtAl2025,
  author    = {Ara{\'u}jo, J. and Dourado, M. C. and Protti, F. and Sampaio, R. M.},
  title     = {Introduction to Graph Convexity: An Algorithmic Approach},
  note      = {Translated from the Portuguese},
  series    = {Latin American Mathematics Series. SBMAC Collection on Applied and Computational Mathematics},
  publisher = {Springer},
  address   = {Cham},
  year      = {2025}
}

@article{Caratheodory1911,
  author  = {Carath{\'e}odory, C.},
  title   = {{\"U}ber den Variabilit{\"a}tsbereich der Fourier'schen Konstanten von positiven harmonischen Funktionen},
  journal = {Rendiconti del Circolo Matematico di Palermo},
  volume  = {32},
  number  = {1},
  year    = {1911},
  pages   = {193--217}
}

@inproceedings{CentenoDouradoEtAl2009,
  author    = {Centeno, C. C. and Dourado, M. C. and Szwarcfiter, J. L.},
  title     = {On the convexity of paths of length two in undirected graphs},
  booktitle = {DIMAP Workshop on Algorithmic Graph Theory},
  editor    = {Koster, A. and others},
  series    = {Electronic Notes in Discrete Mathematics},
  volume    = {32},
  year      = {2009},
  pages     = {11--18}
}

@article{ChangatMathew1999,
  author  = {Changat, M. and Mathew, J.},
  title   = {On triangle path convexity in graphs},
  journal = {Discrete Mathematics},
  volume  = {206},
  year    = {1999},
  pages   = {91--95}
}

@article{ChangatMulderEtAl2005,
  author  = {Changat, M. and Mulder, H. M. and Sierksma, G.},
  title   = {Convexities related to path properties on graphs},
  journal = {Discrete Mathematics},
  volume  = {290},
  year    = {2005},
  pages   = {117--131}
}

@article{Chen2009,
  author  = {Chen, N.},
  title   = {On the approximability of influence in social networks},
  journal = {SIAM Journal on Discrete Mathematics},
  volume  = {23},
  number  = {3},
  year    = {2009},
  pages   = {1400--1415}
}

@article{ErdosFriedEtAl1972,
  author  = {Erd{\H{o}}s, P. and Fried, E. and Hajnal, A. and Milner, E. C.},
  title   = {Some remarks on simple tournaments},
  journal = {Algebra Universalis},
  volume  = {2},
  year    = {1972},
  pages   = {238--245}
}

@incollection{Harary1984,
  author    = {Harary, F.},
  title     = {Convexity in graphs: Achievement and avoidance games},
  booktitle = {Convexity and Graph Theory},
  editor    = {Rosenfeld, M. and Zaks, J.},
  series    = {North-Holland Mathematics Studies},
  volume    = {87},
  publisher = {North-Holland},
  address   = {Amsterdam},
  year      = {1984},
  pages     = {323}
}

@article{HararyNieminen1981,
  author  = {Harary, F. and Nieminen, J.},
  title   = {Convexity in graphs},
  journal = {Journal of Differential Geometry},
  volume  = {16},
  number  = {2},
  year    = {1981},
  pages   = {185--190}
}

@article{Helly1923,
  author  = {Helly, E.},
  title   = {{\"U}ber Mengen konvexer K{\"o}rper mit gemeinschaftlichen Punkten},
  journal = {Jahresbericht der Deutschen Mathematiker-Vereinigung},
  volume  = {32},
  year    = {1923},
  pages   = {175--176}
}

@article{Jamison1981,
  author  = {Jamison, R. E.},
  title   = {Partition numbers for trees and ordered sets},
  journal = {Pacific Journal of Mathematics},
  volume  = {96},
  number  = {1},
  year    = {1981},
  pages   = {115--140}
}

@incollection{Jamison1982,
  author    = {Jamison, R. E.},
  title     = {A perspective on abstract convexity: Classifying alignments by varieties},
  booktitle = {Convexity and Related Combinatorial Geometry},
  editor    = {Kay, D. C. and Breen, M.},
  publisher = {Marcel Dekker},
  address   = {New York},
  year      = {1982}
}

@article{Levi1951,
  author  = {Levi, F. W.},
  title   = {On Helly's theorem and the axioms of convexity},
  journal = {Journal of the Indian Mathematical Society},
  volume  = {15},
  year    = {1951},
  pages   = {65--76}
}

@incollection{Mulder2008,
  author    = {Mulder, H. M.},
  title     = {Transit functions on graphs (and posets)},
  booktitle = {Convexity in Discrete Structures},
  editor    = {Changat, M. and others},
  series    = {Ramanujan Mathematical Society Lecture Notes Series},
  volume    = {5},
  publisher = {Ramanujan Mathematical Society},
  address   = {Mysore},
  year      = {2008},
  pages     = {117--130}
}

@article{Radon1921,
  author  = {Radon, J.},
  title   = {Mengen konvexer K{\"o}rper, die einen gemeinsamen Punkt enthalten},
  journal = {Mathematische Annalen},
  volume  = {83},
  number  = {1--2},
  year    = {1921},
  pages   = {113--115}
}

@incollection{Sierksma1984,
  author    = {Sierksma, G.},
  title     = {Exchange properties of convexity spaces},
  booktitle = {Convexity and Graph Theory, Proc. Conf., Israel, 1981},
  series    = {Annals of Discrete Mathematics},
  volume    = {20},
  year      = {1984},
  pages     = {293--305}
}

@book{vanDeVel1993,
  author    = {van de Vel, M. L.},
  title     = {Theory of Convex Structures},
  publisher = {North-Holland},
  address   = {Amsterdam},
  year      = {1993}
}

\appendix

\counterwithin{theorem}{section}
\counterwithin{lemma}{section}
\counterwithin{proposition}{section}
\counterwithin{corollary}{section}
\counterwithin{figure}{section}
\counterwithin{table}{section}

\section{Characterization of extremal graphs for Theorem~\ref{T:Main} and Corollary~\ref{C:Main}}\label{A:new12}

\begin{proof}[Proof of Theorem~\ref{T:New1}]\hfill\newline
\textbf{(\(\Leftarrow\))}
Let $S$ be the leaf set of the full binary tree $T$, and let $r$ be its root. Fix a leaf $a\in S$ and consider the unique $a$--$r$ path $P$ in $T$. We claim that no vertex of $P$ belongs to $\Hull{S\setminus\{a\}}$. Otherwise let $q$ be the first vertex of $P$ to enter during the iterative hull process. Every triangle of $G$ containing $q$ is either the triangle formed by $q$ and its two children, or (when $q\ne r$) the triangle formed by $q$, its parent, and its sibling. In either case, for that triangle to generate $q$, another vertex of $P$ must already have entered the hull, contradicting the choice of $q$. Thus $r\notin\Hull{S\setminus\{a\}}$ for every $a\in S$. On the other hand, $r\in\Hull{S}$ by firing the triangles bottom-up. Hence $S$ is Carath\'eodory independent.

\medskip
\noindent\textbf{(\(\Rightarrow\))}
Induction on \(t = t(G)\).

We prove the slightly stronger statement that, whenever $S$ is a Carath\'eodory independent set with $|S|=t(G)+1$, the full binary tree $T$ may be chosen so that its leaf set is exactly $S$.

If \(t = 0\) then \(\card{S} = 1\); take \(\mathcal{T}\) to be the single-leaf tree. Conditions (i) and (ii) hold vacuously.

If $t=1$, then, since $G$ is connected and every edge lies on a triangle,
$G=K_3$. Since $|S|=2$, take the two vertices of $S$ as the leaves and the
remaining vertex as the root. The sibling edge gives the third edge of $K_3$,
so the strengthened statement holds.

Suppose \(t \geq 2\).
Since \(\card{S} = t+1\), we have \(\card{\mathcal{E}} \geq 1\) (as \(\card{\mathcal{E}} = 0\) gives \(\card{S} \leq t\)).
We show that Case~1 of the proof of Theorem~\ref{T:Main} applies to some exceptional component.

Since $|S|=t(G)+1$, the revised Case~2 of the proof of Theorem~\ref{T:Main} cannot occur: Case~2 gives the stronger inequality $|S|\le t(G)$. Hence, Case~1 applies to some exceptional component $H$. Let $u$ be the corresponding leaf, let $v$ be its neighbour in $H$, and let $x$ be the unique apex of $\triangle uvx$; thus $xv$ lies on no $4$-cycle in $G-u$.

Form $G'$ and $S'$ as in Case~1. By Lemma~\ref{L:CI-preservation}, $S'$ is Carath\'eodory independent in $G'$, and the Case~1 triangle count gives
\[
t(G')=t(G)-1.
\]
Therefore
\[
|S'|=|S|-1=t(G)=t(G')+1.
\]
Moreover, the reduction preserves the hypotheses of Theorem~5: $G'$ is connected, and every edge of $G'$ lies on a triangle. Indeed, $\triangle uvx$ is the only triangle containing $u$, so every edge incident with $u$ is one of $uv,ux$, and deleting $u$ does not disconnect the graph because $vx$ remains. Every other surviving edge retains a triangle after the deletion and contraction.

Thus the strengthened inductive hypothesis applies to $(G',S')$. We obtain a full binary tree $\mathcal T'$ satisfying the required
structural conditions and having leaf set exactly $S'$. In particular,
$\bar v\in S'$ is a leaf of $\mathcal T'$.

We claim that $x$ belongs to a triangle of $G$ other than
$\triangle uvx$. Otherwise, since $\triangle uvx$ is also the only
triangle containing $u$, deleting $u$ can affect the hull only at
$u$ and $x$. Let
\[
q\in \Hull{S}\setminus\bigcup_{p\in S}\Hull{S\setminus\{p\}}
\]
witness the Carath\'eodory independence of $S$. Then $q\in\{u,x\}$.
However, $q\neq u$, since $u\in S\setminus\{p\}$ for every
$p\in S\setminus\{u\}$; and $q\neq x$, since $|S|\ge3$, so for some
$p\in S\setminus\{u,v\}$ the vertices $u,v\in S\setminus\{p\}$
generate $x$. This is a contradiction.

Thus $x$ belongs to a triangle other than $\triangle uvx$. Its image
in $G'$ is a triangle containing $\bar v$. Since $\bar v$ is a leaf
of $\mathcal T'$, it belongs to exactly one triangle of $G'$. Hence
the unique triangle of $G'$ containing $\bar v$ is the image of a
triangle of $G$ containing $x$.

Construct \(\mathcal{T}\) from \(\mathcal{T}'\) by replacing leaf \(\bar{v}\) with an internal node labelled \(x\), with two new leaf children \(u\) and \(v\).
The leaf set of $\mathcal T$ is
\[
(S'\setminus\{\bar v\})\cup\{u,v\}
=(S\setminus\{u,v\})\cup\{u,v\}=S.
\]
Since \(\{x,u,v\}\) induces the triangle \(\triangle uvx\) in \(G\), the new internal node \(x\) with children \(u\) and \(v\) satisfies condition~(i).
The reduction from $G$ to $G'$ removes exactly the triangle
$\triangle uvx$ and creates no new triangles. Hence it induces a
bijection between the triangles of $G$ other than $\triangle uvx$
and the triangles of $G'$. As shown above, the unique triangle of $G'$ containing $\bar v$
is the image of a triangle of $G$ containing $x$. Thus, replacing
$\bar v$ by $x$ preserves the triangle represented by its parent
in $\mathcal T'$, while the new internal node $x$ with children
$u$ and $v$ represents $\triangle uvx$. All other internal nodes
of $\mathcal T'$ and their corresponding triangles are unchanged.
Hence $\mathcal T$ satisfies conditions~(i) and~(ii).
\end{proof}

Before proving the characterization of the extremal graphs for the exchange number, we first record a simple observation about the graphs appearing in
Theorem~\ref{T:New1}.

\begin{lemma}\label{L:tree-hulls}
Let $F$ be obtained from a rooted full binary tree $\mathcal T$ by
adding an edge between every pair of siblings, let $L$ be the leaf
set of $\mathcal T$, and let $r$ be its root.  For every leaf
$a\in L$,
\[
 V(F)\setminus \Hull{L\setminus\{a\}}
\]
is exactly the vertex set of the $a$--$r$ path in $\mathcal T$.
Consequently, $r$ is the unique vertex in
\[
 \Hull{L}\setminus\bigcup_{a\in L}\Hull{L\setminus\{a\}}.
\]
Moreover, if $z$ is a non-leaf vertex of $\mathcal T$ and $a$ is a
leaf in the subtree rooted at $z$, then
\[
 r\in \Hull{(L\cup\{z\})\setminus\{a\}}.
\]
\end{lemma}

\begin{proof}
Fix $a\in L$.  Every vertex not on the $a$--$r$ path is generated
bottom-up from the leaves in its subtree, all of which are still
present.  No vertex on the $a$--$r$ path can be generated.  Indeed,
if a path vertex were the first one to enter the hull, then every
triangle capable of generating it would contain another vertex of
the same path which had not yet entered: for its child triangle this
is the child on the path, and for its parent triangle this is the
parent on the path.  The root is handled by its child triangle, and
the missing leaf $a$ cannot enter first.  This proves the first
assertion.  The intersection of all root--leaf
paths is just the root, which proves the second assertion.  For the
last assertion, the presence of $z$ bypasses the missing part of the
$a$--$r$ path below $z$; from $z$ upward, each successive ancestor is
generated using the already generated sibling branch.
\end{proof}

\begin{proof}[Proof of Theorem~\ref{T:New2}]
The case $t(G)=0$ is immediate, so assume $t(G)\ge1$.
Suppose first that $e_{\Tri}(G)=t(G)+2$.  Let $S$ be an exchange
independent set of this size.  Choose $p\in S$ and
\[
 q\in \Hull{S\setminus\{p\}}
 \setminus
 \bigcup_{x\in S\setminus\{p\}}\Hull{S\setminus\{x\}}.
\]
Put $R=S\setminus\{p\}$.  Then $R$ is Carath\'eodory independent,
since for every $x\in R$,
\[
 \Hull{R\setminus\{x\}}\subseteq \Hull{S\setminus\{x\}}.
\]
Thus $|R|=t(G)+1$, and the general Carath\'eodory bound is attained.

We first show that $G$ has an edge which lies on no triangle.  If
every edge lay on a triangle, the strengthened form of
Theorem~\ref{T:New1} would give a full binary tree $\mathcal T$ with
leaf set exactly $R$.  By Lemma~\ref{L:tree-hulls}, the witness $q$
must be the root of $\mathcal T$.  Since $p\notin R$, the vertex $p$
is an internal vertex.  Choose a leaf $a$ below $p$.  Then
Lemma~\ref{L:tree-hulls} gives
\[
 q\in \Hull{S\setminus\{a\}},
\]
contrary to the choice of $q$.  Hence there is a unique edge $f$ of
$G$ which lies on no triangle.

Delete $f$.  This does not change triangle convexity.  Let $C$ be
the component of $G-f$ containing $q$.  Every vertex of $R$ belongs
to $C$: otherwise deleting a vertex of $R$ outside $C$ would not
affect whether $q$ belongs to the hull.  Therefore
\[
 t(G)+1=|R|\le t(C)+1\le t(G)+1,
\]
so $t(C)=t(G)$ and $R$ is extremal in $C$.  In particular, every
triangle of $G$ lies in $C$, and Theorem~\ref{T:New1} applies to $C$.
If $p\in V(C)$, the same argument using Lemma~\ref{L:tree-hulls}
again contradicts the exchange independence of $S$.  Hence
$p\notin V(C)$.

Every edge of $G-f$ outside $C$ would lie on a triangle, but all
triangles lie in $C$.  Thus the other component of $G-f$ is the
single vertex $p$.  Consequently $f$ is a pendant edge joining $p$
to $C$, and $C$ is obtained from a full binary tree by adding an
edge between every pair of siblings.

Conversely, let $C$ be obtained from a full binary tree $\mathcal T$
by adding the sibling edges, and attach a new leaf $p$ to any vertex
of $C$.  Let $L$ be the leaf set of $\mathcal T$ and let $r$ be its
root.  Then $|L|=t(G)+1$, and by Lemma~\ref{L:tree-hulls},
\[
 r\in \Hull{L}\setminus
 \bigcup_{a\in L}\Hull{L\setminus\{a\}}.
\]
The new vertex $p$ lies on no triangle, so
$S=L\cup\{p\}$ is exchange independent with distinguished vertex
$p$ and witness $r$.  Hence
\[
 e_{\Tri}(G)\ge |S|=t(G)+2.
\]
The reverse inequality is the general bound, so equality holds.
\end{proof}

\section{Block graphs}\label{A:new345}

For the block-graph arguments, it is convenient to isolate one
reduction that will be used repeatedly.

\begin{lemma}\label{L:block-skeleton}
Let $H$ be a connected block graph with no $K_2$-blocks, and let
$A\subseteq V(H)$ with $|A|\ge2$.  Suppose that
\[
 q\in \Hull{A}
 \setminus
 \bigcup_{a\in A}\Hull{A\setminus\{a\}}.
\]
Then there is a connected subgraph $F\subseteq H$ containing
$A\cup\{q\}$ such that
\begin{enumerate}
\item every block of $F$ is a triangle;
\item $F$ contains at most one triangle from each block of $H$; and
\item $A$ remains Carath\'eodory independent in $F$, with the
same witness $q$.
\end{enumerate}
In particular, $t(F)\le b(H)$.
\end{lemma}

\begin{proof}
Root the block--cut tree of $H$ towards $q$, and trace backwards a
fixed derivation of $q$ from $A$.  Since every block of $H$ is a
clique, whenever the derivation passes through a block $B$ towards
$q$, only two already available vertices of $B$ are needed to
generate the vertex of $B$ lying in the direction of the root.  Keep
only the triangle formed by these three vertices.  If either of the
two supporting vertices was itself generated from a block farther
from the root, repeat the same procedure there.

Because the block--cut graph is a tree, this recursion never returns
to a block already used.  Hence at most one triangle is retained
from each block, and the retained triangles form a connected
subgraph $F$ in which the same derivation still produces $q$ from
$A$.  Every $a\in A$ must occur in this derivation, for otherwise
$q\in\Hull{A\setminus\{a\}}$.  Thus $A\subseteq V(F)$.  Finally, since $F\subseteq H$, the hull of $A\setminus\{a\}$
computed in $F$ is contained in $\Hull{A\setminus\{a\}}$
computed in $H$, for every $a\in A$.  Thus the same vertex $q$
witnesses the Carath\'eodory independence of $A$ in $F$.
\end{proof}

\begin{proof}[Proof of Theorem~\ref{T:New3}]
Delete all $K_2$-blocks of $G$.  Since no such edge lies on a
triangle, triangle closure acts independently on the resulting
components.

Let $S$ be Carath\'eodory independent, with witness $q$.  If
$|S|=1$, the required bound is immediate.  Hence assume $|S|\ge2$.
Every vertex of $S$ lies in the same component as $q$, for otherwise
removing a vertex in another component would not affect the
membership of $q$ in the hull.  Thus $S\subseteq V(H_0)$ for some
$H_0$.  By Lemma~\ref{L:block-skeleton}, there is a triangle-block
subgraph $F\subseteq H_0$ in which $S$ remains Carath\'eodory
independent and
\[
 t(F)\le b(H_0).
\]
Hence
\[
 |S|\le t(F)+1\le b(H_0)+1,
\]
which proves the first bound.

Now let $S$ be exchange independent.  Choose $p\in S$ and
\[
 q\in \Hull{S\setminus\{p\}}
 \setminus
 \bigcup_{x\in S\setminus\{p\}}\Hull{S\setminus\{x\}},
\]
and put $R=S\setminus\{p\}$.  As above, $R$ is Carath\'eodory
independent with witness $q$.  If $|R|=1$, the required exchange
bound is immediate.  Hence assume $|R|\ge2$, and let $H_0$ be the
component containing $R$ and $q$.

Suppose first that $p\notin V(H_0)$.  The first $K_2$-block on the
path from $H_0$ to $p$, together with $H_0$, is contained in a
subgraph $H_1$ of the type appearing in the statement, with
\[
 b(H_1)=b(H_0)+1.
\]
Therefore
\[
 |S|=|R|+1\le b(H_0)+2=b(H_1)+1.
\]

It remains to consider $p\in V(H_0)$.  If
$|R|\le b(H_0)$, then immediately
\[
 |S|\le b(H_0)+1.
\]
Suppose instead that $|R|=b(H_0)+1$.  Equality must then hold
throughout the Carath\'eodory argument above.  Lemma~\ref{L:block-skeleton}
therefore gives a subgraph $F$ with exactly one triangle from every
block of $H_0$ and
\[
 |R|=t(F)+1.
\]
By the strengthened form of Theorem~\ref{T:New1}, $F$ is obtained
from a full binary tree $\mathcal T$ whose leaf set is exactly $R$.
Since $q$ remains a witness in $F$, Lemma~\ref{L:tree-hulls} shows
that $q$ is the root of $\mathcal T$.

If $p\in V(F)$, then $p$ is an internal vertex of $\mathcal T$.
Choosing a leaf $a$ below $p$ and applying
Lemma~\ref{L:tree-hulls} gives
$q\in\Hull{S\setminus\{a\}}$, a contradiction.  If
$p\notin V(F)$, let $B$ be the block of $H_0$ containing $p$ and let
$\triangle zxy$ be the unique triangle of $F$ contained in $B$,
where $z$ is the parent of $x$ and $y$ in $\mathcal T$.  Choose a
leaf $a$ below $x$.  In \(\Hull{S\setminus\{a\}}\), the vertex $y$
is generated from its intact subtree, while $p$ is present from the
start.  Since $B$ is a clique, $p,y,z$ form a triangle, and hence
$z$ is generated.  Proceeding upward in $\mathcal T$ now generates
$q$, again a contradiction.  Thus
\[
 |S|\le b(H_0)+1.
\]
Combining the two cases proves the exchange-number bound.
\end{proof}

\begin{proof}[Proof of Theorem~\ref{T:New4}]
If $b(G)=0$, then $G$ consists of a single vertex and the statement
is immediate.  Suppose first that $c_{\Tri}(G)=b(G)+1$, and let $S$
be a Carath\'eodory independent set of this size with witness $q$.
Lemma~\ref{L:block-skeleton} gives a connected triangle-block
subgraph $F\subseteq G$ such that $S$ remains Carath\'eodory
independent in $F$ and
\[
 b(G)+1=|S|\le t(F)+1\le b(G)+1.
\]
Thus $t(F)=b(G)$, so $F$ contains exactly one triangle from every
block of $G$, and $c_{\Tri}(F)=t(F)+1$.  By
Theorem~\ref{T:New1}, $F$ is obtained from a full binary tree with
leaf set $S$ by adding an edge between every pair of siblings.
Every block of $G$ is a clique containing exactly one triangle block
of $F$, so $G$ is obtained from $F$ by enlarging these blocks one
vertex at a time.

Conversely, suppose $G$ is obtained in this way from a graph $F$
associated with a full binary tree $\mathcal T$, and let $L$ be the
leaf set of $\mathcal T$.  The added vertices are not cut vertices
and are not initially present in $L$.  Inside an enlarged block,
such a vertex can enter a hull only after two vertices of that block
have already entered.  Hence enlarging the blocks does not change
the closure process on the vertices of $F$.  By
Lemma~\ref{L:tree-hulls}, the root of $\mathcal T$ therefore still
witnesses that $L$ is Carath\'eodory independent in $G$.  Since
\[
 |L|=b(G)+1,
\]
Theorem~\ref{T:New3} gives $c_{\Tri}(G)=b(G)+1$.
\end{proof}

\begin{proof}[Proof of Theorem~\ref{T:New5}]
Let $p$ be the leaf endpoint of the unique pendant $K_2$-block, and
put $H=G-p$.  Then $H$ is a connected block graph with no
$K_2$-blocks and
\[
 b(H)=b(G)-1.
\]
Since $p$ lies on no triangle, triangle closure on $V(H)$ is
unchanged and
\[
 c_{\Tri}(G)=c_{\Tri}(H).
\]
Moreover,
\[
 e_{\Tri}(G)=c_{\Tri}(H)+1.
\]
Indeed, the upper bound follows from Sierksma's inequality, while for
the reverse inequality one takes a maximum Carath\'eodory
independent set $R$ of $H$ and observes that $R\cup\{p\}$ is exchange
independent with distinguished vertex $p$.

Consequently,
\[
 e_{\Tri}(G)=b(G)+1=c_{\Tri}(G)+1
\]
holds if and only if
\[
 c_{\Tri}(H)=b(H)+1.
\]
The result now follows directly from Theorem~\ref{T:New4}, followed
by the attachment of the pendant leaf $p$.
\end{proof}

\section{\(k\)-trees}\label{A:new6}

We use the following consequence of the proof of the known result for
$2$-connected chordal graphs.  Recall that an edge $uv$ is
\emph{dominating} if $N[u]\cup N[v]=V(G)$.

\begin{lemma}\label{L:dominating-edge}
Let $G$ be a $2$-connected chordal graph.  Then
\[
 e_{\Tri}(G)=2
\]
if and only if every edge of $G$ is dominating.
\end{lemma}

\begin{proof}
In a $2$-connected chordal graph, every pair of adjacent vertices is
a hull set.  Suppose first that every edge is dominating.  Any
exchange independent set of size at least three must contain an
edge, since an independent set is already closed under triangle
convexity and is exchange dependent.  Let $ab$ be such an edge.  For any third vertex $c$ in the set, $c$ is adjacent
to $a$ or $b$; say $bc\in E(G)$.  Then
\[
 \Hull{\{a,b\}}=\Hull{\{b,c\}}=V(G),
\]
so two distinct one-vertex deletions already have full hull.  Hence
the set is exchange dependent, and $e_{\Tri}(G)=2$.

Conversely, suppose that some edge $ab$ is not dominating, and let
$c$ be adjacent to neither $a$ nor $b$.  Then
\[
 \Hull{\{a,b\}}=V(G),
 \qquad
 \Hull{\{a,c\}}=\{a,c\},
 \qquad
 \Hull{\{b,c\}}=\{b,c\}.
\]
Since $G$ is connected and $c$ is adjacent to neither $a$ nor $b$,
there is a vertex outside $\{a,b,c\}$.  Any such vertex witnesses
that $\{a,b,c\}$ is exchange independent.  Since the exchange number
of a $2$-connected chordal graph is at most $3$, we obtain
$e_{\Tri}(G)=3$.
\end{proof}

\begin{proof}[Proof of Theorem~\ref{T:k-tree}]
The case $n=k$ is immediate, so assume $n\ge k+1$.  Then $G$ is a
$2$-connected chordal graph, and by Lemma~\ref{L:dominating-edge} it
suffices to characterize the $k$-trees in which every edge is
dominating.

We prove by induction on $n$ that every such $k$-tree has the form
$K_k\vee\overline{K}_{n-k}$.  Let $v$ be a last vertex in a
$k$-tree construction of $G$, and let $Q=N(v)$; thus $Q$ is a
$k$-clique.  The graph $G-v$ is again a $k$-tree, and every edge of
$G-v$ is still dominating in $G-v$.  By induction,
\[
 G-v=C\vee I,
\]
where $C$ is a $k$-clique and $I$ is independent.

If $Q=C$, then
\[
 G=C\vee(I\cup\{v\}),
\]
as required.  Suppose $Q\ne C$.  Since $Q$ is a $k$-clique in
$C\vee I$, it has the form
\[
 Q=(C\setminus\{c\})\cup\{y\}
\]
for some $c\in C$ and $y\in I$.  If there were another vertex
$z\in I\setminus\{y\}$, then $z$ would be adjacent to neither $v$
nor $y$, contradicting the fact that the edge $vy$ is dominating.
Hence $I=\{y\}$.  Now $c$ and $v$ are nonadjacent and both are
adjacent to every vertex of $Q$, so
\[
 G=Q\vee\overline{K}_2,
\]
which is again of the required form.

Conversely, every edge of $K_k\vee\overline{K}_{n-k}$ is dominating.
Hence Lemma~\ref{L:dominating-edge} gives $e_{\Tri}(G)=2$.  Every
other $k$-tree has exchange number $3$ by the known bound for
$2$-connected chordal graphs.
\end{proof}

\section{Figures}\label{A:figures}

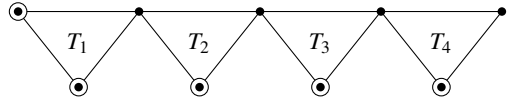
\begin{figure}
\centering
\begin{tikzpicture}[every node/.style={font=\small}]

\coordinate (v0) at (0,0);
\coordinate (v1) at (1.6,0);
\coordinate (v2) at (3.2,0);
\coordinate (v3) at (4.8,0);
\coordinate (v4) at (6.4,0);

\coordinate (u1) at (0.8,-1.0);
\coordinate (u2) at (2.4,-1.0);
\coordinate (u3) at (4.0,-1.0);
\coordinate (u4) at (5.6,-1.0);

\draw (v0)--(v1)--(u1)--(v0);
\draw (v1)--(v2)--(u2)--(v1);
\draw (v2)--(v3)--(u3)--(v2);
\draw (v3)--(v4)--(u4)--(v3);

\node at (0.8,-0.42) {$T_1$};
\node at (2.4,-0.42) {$T_2$};
\node at (4.0,-0.42) {$T_3$};
\node at (5.6,-0.42) {$T_4$};

\foreach \x in {v0,u1,u2,u3,u4}
    \draw[fill=white] (\x) circle (3.3pt);

\foreach \x in {v0,v1,v2,v3,v4,u1,u2,u3,u4}
    \fill (\x) circle (1.6pt);

\end{tikzpicture}
\caption{A block graph formed by a chain of four triangles \(T_1,T_2,T_3,T_4\). The circled vertices form a largest Carath\'eodory independent set.}
\label{fig:triangle-chain-four}
\end{figure}

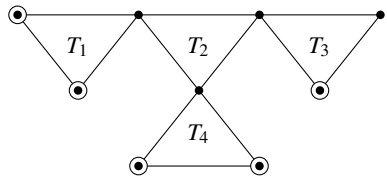
\begin{figure}
\centering
\begin{tikzpicture}[every node/.style={font=\small}]

\coordinate (v0) at (0,0);
\coordinate (v1) at (1.6,0);
\coordinate (v2) at (3.2,0);
\coordinate (v3) at (4.8,0);

\coordinate (u1) at (0.8,-1.0);
\coordinate (u2) at (2.4,-1.0);
\coordinate (u3) at (4.0,-1.0);

\coordinate (w1) at (1.6,-2.0);
\coordinate (w2) at (3.2,-2.0);

\draw (v0)--(v1)--(u1)--(v0);
\draw (v1)--(v2)--(u2)--(v1);
\draw (v2)--(v3)--(u3)--(v2);
\draw (w1)--(w2)--(u2)--(w1);

\node at (0.8,-0.42) {$T_1$};
\node at (2.4,-0.42) {$T_2$};
\node at (4.0,-0.42) {$T_3$};
\node at (2.4,-1.57) {$T_4$};

\foreach \x in {v0,u1,u3,w1,w2}
    \draw[fill=white] (\x) circle (3.3pt);

\foreach \x in {v0,v1,v2,v3,u1,u2,u3,w1,w2}
    \fill (\x) circle (1.6pt);

\end{tikzpicture}
\caption{A block graph obtained from a chain of three triangles by attaching a fourth triangle at the degree-\(2\) vertex of the middle triangle. The circled vertices form a largest Carath\'eodory independent set.}
\label{fig:block-counterexample}
\end{figure}

\begin{figure}
\centering
\begin{tikzpicture}[scale=1.0]
\begin{scope}[xshift=0cm]
  \node[vtxlbl] at (0.7,2.5) {$G$};
  \node[vtxlbl] at (-2.4,0) {$\cdots$};
  \node[inS, label={[vtxlbl]below:$w$}]  (w) at (-1.8,0) {};
  \node[inS, label={[vtxlbl]below:$v$}]  (v) at (0,0)   {};
  \node[inS, label={[vtxlbl]below:$u$}]  (u) at (2,0)   {};
  \node[notS,label={[vtxlbl]above:$x$}]  (x) at (1,1.7) {};
  \draw[Hedge] (w) -- (v) -- (u);    
  \draw[Hedge,dashed] (-2.1,0) -- (w); 
  \draw (u) -- (x) -- (v);           
  \node[vtxlbl] at (1,0.55) {$\triangle uvx$};
\end{scope}
\draw[->,thick] (3.2,0.85) -- (4.0,0.85);
\node[vtxlbl] at (3.6,1.15) {\small contract $xv$};
\node[vtxlbl] at (3.6,0.55) {\small delete $u$};
\begin{scope}[xshift=5.0cm]
  \node[vtxlbl] at (0.5,2.5) {$G'$};
  \node[vtxlbl] at (-2.4,0) {$\cdots$};
  \node[inS, label={[vtxlbl]below:$w$}]   (w2) at (-1.8,0) {};
  \node[inS, label={[vtxlbl]below:$\bar{v}$}] (vb) at (0,0) {};
  \draw[Hedge] (w2) -- (vb);
  \draw[Hedge,dashed] (-2.1,0) -- (w2);
\end{scope}
\end{tikzpicture}
\caption{The leaf $u \in S$ and its apex $x \notin S$ are removed; the edge $xv$ is contracted to $\bar{v} \in S'$. Thick edges lie in the exceptional component $H$.}
\end{figure}
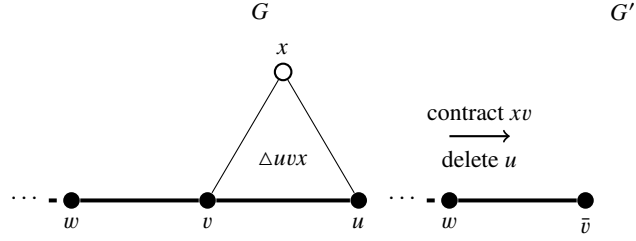

\begin{figure}
\centering
\begin{tikzpicture}[scale=1.0]
  \node[inS, label={[vtxlbl]below:$u$}] (u) at (2.5,0) {};
  \node[inS, label={[vtxlbl]below:$v$}] (v) at (0,0)   {};
  \node[inS, label={[vtxlbl]below:$w$}] (w) at (-2.5,0){};
  \node[vtxlbl] at (-3.5,0) {$\cdots$};
  \draw[Hedge,dashed] (-3.1,0) -- (w);
  \draw[Hedge] (u) -- (v) -- (w);
  \node[notS,label={[vtxlbl]right:$x$}] (x) at (1.5,1.8) {};
  \node[notS,label={[vtxlbl]left:$y$}]  (y) at (-1.5,1.8){};
  \draw (u) -- (x) -- (v);
  \draw (v) -- (y) -- (w);
  \node[notS,label={[vtxlbl]above:$z$}] (z) at (0,3.4) {};
  \node[notS,label={[vtxlbl]right:$a$}] (a) at (3.0,3.0){};
  \node[notS,label={[vtxlbl]left:$b$}]  (b) at (-3.0,3.0){};
  \draw (x) -- (z) -- (a) -- (x);
  \draw (y) -- (z) -- (b) -- (y);
  \node[vtxlbl] at (2.3,2.5) {\small$\triangle xza$};
  \node[vtxlbl] at (-2.3,2.5) {\small$\triangle yzb$};
  \node[vtxlbl] at (1.8,0.7) {\small$\triangle uvx$};
  \node[vtxlbl] at (-1.8,0.7) {\small$\triangle vwy$};
\end{tikzpicture}
\caption{Thick edges lie in the exceptional component $H$; $u$ is a chosen leaf. The 4-cycle $xvyzx$ accounts for the edges $xv$ (in $\triangle uvx$), $vy$ (in $\triangle vwy$), $yz$ (in $\triangle yzb$), and $zx$ (in $\triangle xza$). Vertices $a$ and $b$ are the spare triangle vertices whose membership in $S$ drives the subcase analysis.}
\end{figure}
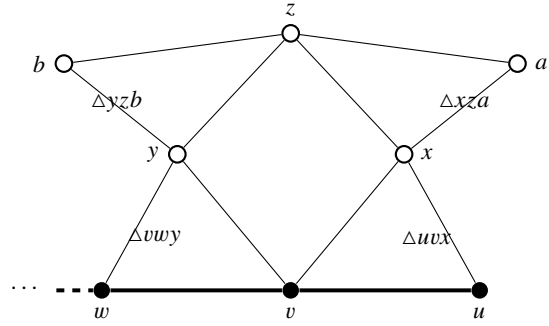

\begin{figure}
\centering
\begin{tikzpicture}[scale=1.1]
  \fill[blue!12] (-2.5,0) -- (-0.5,0) -- (0,2) -- cycle;
  \fill[blue!12] (0.5,0) -- (2.5,0) -- (0,2) -- cycle;
  \draw[dashed, gray!70, rounded corners=5pt]
    (-3.2,-0.5) rectangle (-0.1,0.5);
  \node[gray, font=\small, anchor=south west] at (-3.2,-0.5) {$H$};
  \draw[dashed, gray!70, rounded corners=5pt]
    (0.1,-0.5) rectangle (3.2,0.5);
  \node[gray, font=\small, anchor=south east] at (3.2,-0.5) {$H'$};
  \draw[line width=1.5pt] (-2.5,0) -- (-0.5,0);
  \draw[line width=1.5pt] (0.5,0) -- (2.5,0);
  \draw (-2.5,0) -- (0,2);
  \draw (-0.5,0) -- (0,2);
  \draw (0.5,0) -- (0,2);
  \draw (2.5,0) -- (0,2);
  \fill (-2.5,0) circle (2.5pt);
  \fill (-0.5,0) circle (2.5pt);
  \fill (0.5,0) circle (2.5pt);
  \fill (2.5,0) circle (2.5pt);
  \draw[fill=white, line width=0.8pt] (0,2) circle (2.5pt);
  \node[below=2pt] at (-2.5,0) {$u$};
  \node[below=2pt] at (-0.5,0) {$v$};
  \node[below=2pt] at (0.5,0) {$u'$};
  \node[below=2pt] at (2.5,0) {$v'$};
  \node[above=2pt] at (0,2) {$x$};
\end{tikzpicture}
\caption{If \(x\) were the apex of exceptional edges \(uv\) in \(H\) and \(u'v'\) in \(H'\) simultaneously, then \(u',v' \in S \setminus \{u\}\) would force \(x\) into \(I_\triangle(S \setminus \{u\})\), and then \(v \in S\setminus\{u\}\) and \(\triangle uvx\) would give \(u \in \Hull{S\setminus\{u\}}\), a contradiction.}
\label{fig:unique-apex}
\end{figure}
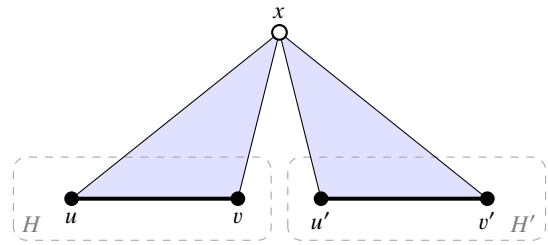

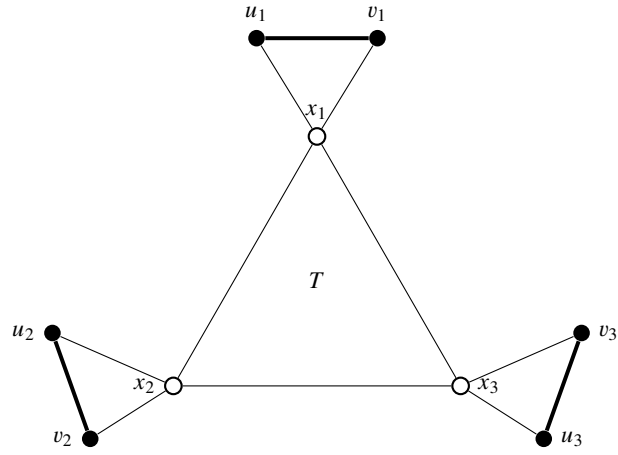
\begin{figure}
\centering
\begin{tikzpicture}[scale=1.0]
  \node[notS,label={[vtxlbl]above:$x_1$}] (x1) at (0,2.2)     {};
  \node[notS,label={[vtxlbl]left:$x_2$}]  (x2) at (-1.9,-1.1) {};
  \node[notS,label={[vtxlbl]right:$x_3$}] (x3) at (1.9,-1.1)  {};
  \draw (x1) -- (x2) -- (x3) -- (x1);
  \node[vtxlbl] at (0,0.3) {$T$};
  \node[inS,label={[vtxlbl]above:$u_1$}] (u1) at (-0.8,3.5) {};
  \node[inS,label={[vtxlbl]above:$v_1$}] (v1) at (0.8,3.5)  {};
  \draw[Hedge] (u1) -- (v1);
  \draw (u1) -- (x1) -- (v1);
  \node[inS,label={[vtxlbl]left:$u_2$}] (u2) at (-3.5,-0.4) {};
  \node[inS,label={[vtxlbl]left:$v_2$}] (v2) at (-3.0,-1.8) {};
  \draw[Hedge] (u2) -- (v2);
  \draw (u2) -- (x2) -- (v2);
  \node[inS,label={[vtxlbl]right:$v_3$}] (v3) at (3.5,-0.4) {};
  \node[inS,label={[vtxlbl]right:$u_3$}] (u3) at (3.0,-1.8) {};
  \draw[Hedge] (u3) -- (v3);
  \draw (u3) -- (x3) -- (v3);
\end{tikzpicture}
\caption{If all three vertices of $T$ were apices of exceptional edges, then in $\mathrm{Hull}(S\setminus\{u_1\})$ the vertices $x_2$ and $x_3$ enter the first iterate (via $u_2v_2$ and $u_3v_3$), then $T$ forces $x_1$ into the hull, and $\triangle u_1v_1x_1$ returns $u_1$, a contradiction.}
\end{figure}

\end{document}